\documentclass[a4 paper,11pt]{article}
\usepackage[centertags]{amsmath}
\usepackage{amsfonts}
\usepackage{amssymb}
\usepackage{amsthm}
\usepackage{rotating}
\usepackage{graphicx}
\usepackage{amsmath}
\usepackage{latexsym}
\usepackage{amscd}
\usepackage[dvips]{color}
\usepackage{mathrsfs}
\usepackage{slashbox}
\usepackage{inputenc}
\usepackage{tabularx}
\usepackage{makeidx}
\usepackage{epic}
\usepackage{colortbl}
\usepackage[small,normal,bf,up]{caption2}
\usepackage[bookmarksopen]{hyperref}
\usepackage{array,multirow}
\usepackage{longtable}
\usepackage[dvips]{color}
\providecommand{\bysame}{\leavevmode\hbox to3em{\hrulefill}\thinspace}
\providecommand{\MR}{\relax\ifhmode\unskip\space\fi MR }

\providecommand{\href}[2]{#2}
\newtheorem{thm}{Theorem}[section]
\newtheorem{cor}[thm]{Corollary}
\newtheorem{lem}[thm]{Lemma}

\newcommand{\hypergeom}[5]{\mbox{$
_{#1} F_{#2}
\!\!
\left(
\!\!\!\!
\begin{array}{c}
\multicolumn{1}{c}{\begin{array}{c}
#3
\end{array}}\\[1mm]
\multicolumn{1}{c}{\begin{array}{c}
#4
            \end{array}}\end{array}
\!\!\!\!
; \displaystyle{#5}\right)
$}}

\newcommand{\hypergeomq}[6]{\mbox{$
_{#1} \phi _{#2}
\!\!
\left(
\!\!\!\!
\begin{array}{c}
\multicolumn{1}{c}{\begin{array}{c}
#3
\end{array}}\\[1mm]
\multicolumn{1}{c}{\begin{array}{c}
#4
            \end{array}}\end{array}
\!\!\!\!
; \displaystyle{#5}
;\displaystyle{#6}\right)
$}}

\numberwithin{equation}{section}
\let\pn=\\
\def\finpr{\hfill \hbox{
\vrule height 1.453ex  width 0.093ex  depth 0ex \vrule height
1.5ex  width 1.3ex  depth -1.407ex\kern-0.1ex \vrule height
1.453ex  width 0.093ex  depth 0ex\kern-1.35ex \vrule height
0.093ex  width 1.3ex  depth 0ex}}

\def\itm#1#2{\hbox{\hbox to 5mm{\it (#1)}} \hbox {\vtop{\hsize=144mm#2\vfill\ms}\hfill}}
\def\bib#1#2{\hbox{\hbox to 10mm{[#1] :\hfill} \hfill \hbox to 144mm{\vtop{\hsize=144mm#2\vfill\ms}\hfill}}}
\long\def\tx#1#2{\hbox{\hbox to
94mm{\vtop{\hsize=90mm#1\vfill}\hfill} \hfill\hbox to
74mm{\vtop{\hsize=70mm#2\vfill}\hfill}}\filbreak}
\everymath{\displaystyle}
\title{On inversion and connection coefficients for basic hypergeometric polynomials.}
 \author{
{Hamza Chaggara \footnote{The corresponding author: hamza.chaggara@ipeim.rnu.tn} and Mohamed Mabrouk}\\\\
{\small \'{E}cole Sup\'{e}rieure des Sciences et de Technologie, Sousse University, TUNISIA.}
}
\date{}

\begin{document}\thispagestyle{empty}
 \maketitle
 \hrule
\begin{abstract}
 {\footnotesize  In this paper, we propose a general method to express explicitly the inversion and the connection coefficients between two basic hypergeometric polynomial sets. As application, we consider some $d$-orthogonal
basic hypergeometric polynomials and we derive
expansion formulae corresponding to all the families within the $q$-Askey scheme.}
\end{abstract}
\hrule
 \section{Introduction} Let ${\cal P}$ be the vector
space of polynomials with coefficients in $\mathbb{C}$. A
polynomial sequence $\{P_n\}_{n\geq0}$ in $\cal P$ is
called a {\it polynomial set} if and only if $deg P_n=n$. \\
 Given two polynomial sets $\{P_n\}_{n\geq0}$ and
$\{Q_n\}_{n\geq0}$, the so-called {\it connection problem} between
them asks to find the coefficients $C_m(n)$ in the expression:
\begin{eqnarray}
\label{connection}
 Q_n(x)=\mathop\sum_{m=0}^{n}C_m(n)P_m(x).
\end{eqnarray}
For the particular case $Q_n(x)=x^n$ the connection problem
(\ref{connection}) is called { \it inversion problem} associated
to $ \{P_n(x)\}_{n\geq 0}$.
 For discrete polynomials and basic polynomials, besides the natural basis $\{x^n\}_n$, some other basis can be considered, namely, the Pochhammer basis $\{(x)_n\}_n$, the $q$-shifted factorial basis $\{(x;q)_n\}_n$ or products involving them :
 $$(x)_n=\left\{\begin{array}{lll}
               \prod_{k=0}^{n-1}(x+k) & \textrm{if} & n=1,2,3,\ldots \\
               1 & \textrm{if} & n=0
             \end{array}\right.
              \ \textrm{and}\quad (x;q)_n=\left\{\begin{array}{lll}
               \prod_{k=0}^{n-1}(1-xq^k) & \textrm{if} & n=1,2,3,\ldots \\
               1 & \textrm{if} & n=0
             \end{array}\right.$$
The problem of connecting orthogonal polynomials is of old and recent interest.
The connection
coefficients play an important role in many problems in pure and
applied mathematics especially in combinatorial analysis or in mathematical physics. In fact, some
inversion problems have been solved as one of the steps
 leading to the orthogonality for the corresponding polynomial sets. Moreover
  the use of inversion in order to solve connection problems was considered by Rainville~\cite{Rain}
 (Hermite, Laguerre and Legendre polynomials), by Gasper~\cite{gasper} (classical discrete orthogonal
  polynomials) and then by Area et al. \cite{navima2001} (polynomials within the Askey scheme and its $q$-analogue) and by Foupouagnigni et al. \cite{mama2013,mama2015} (classical continuous, classical discrete and $q$-classical orthogonal polynomials).\\
The connection problem of basic polynomials have been studied by
many authors. A wide variety of methods, based
on specific properties of the involved polynomials, have been devised for computing the connection coefficients.
For classical orthogonal polynomials the connection problem can be recurrently solved using an algorithm (the \textit{Navima}-algorithm) which generates in a systematic way a linear recurrence relation in $m$ for $C_m(n)$.
(See, for
instance,~\cite{navima99,navima2015} and the reference therein). An algorithmic approach to
build and solve recurrence relations for connection coefficients associated to $q$-classical orthogonal polynomials was also given by Lewanowicz \cite{lewanowicz2013} and by Foupouagnigni et al. \cite{mama2012}.\\
An expansion of basic hypergeometric functions in basic hypergeometric functions called Verma formula \cite{verma2} was used by S\'{a}nchez-Ruiz and Dehesa \cite{dehesa2001} and by Area et al. \cite{navima2001} to connect hypergeometric and basic hypergeometric polynomials. A
general method based in lowering operators, dual sequences and generating functions was developed by Ben Cheikh and Chaggara to generate and compute the inversion and connection coefficients for polynomial sets with Boas-Buck generating functions \cite{chaggara2005,chaggara2006a}. The same approach was implemented to \textit{Maple} system  to solve particular connection problem for continuous and discrete classical polynomials \cite{chaggara2007b}.\\
Our aim in this paper is to propose a simple and general method which allows us to compute the inversion and connection coefficients for basic hypergeometric polynomials. The approach we shall propose in this paper does not need particular properties of the polynomials involved in the problem.\\
Consider $\{B_n\}_{n\geq 0}$ a suitable basis of polynomials. To find the coefficients $C_m(n)$ in (\ref{connection}),
we combine the inversion relation
\begin{equation}\label{inversion}
B_n(x)=\mathop\sum_{m=0}^nI_m(n)Q_m(x),
\end{equation}
with the explicit expression
\begin{equation}\label{direct}
 P_n(x)=\mathop\sum_{m=0}^nD_m(n)B_m(x),
\end{equation}
which yields, with sum manipulation, to the representation
\begin{eqnarray}\label{formule connection}
C_m(n)=\mathop\sum_{j=0}^{n-m}D_{j+m}(n)I_m(j+m)=\mathop\sum_{k=m}^{n}D_{k}(n)I_m(k).
\end{eqnarray}
The connection coefficient in (\ref{direct}) can be obtained directly from the  hypergeometric or the basic hypergeometric representation of $P_n$ however the inversion coefficient will be evaluated recurrently using, for this purpose, a general result recently stated by Ben Romdhane \cite{neila2015a}.  \\
We derive the inversion and connection coefficients for the following two classes of basic hypergeometric polynomials:
\begin{equation}\label{basic2}
\hypergeomq{r+1}{s}{q^{-n},(a_r)}{(b_s)}{q}{qx}.
\end{equation}
and
\begin{equation}\label{basic1}
\hypergeomq{r+2}{s}{q^{-n},aq^n,(a_r)}{(b_s)}{q}{qx}.
\end{equation}
The $_r\phi_s$ denotes the basic hypergeometric series or $q$-hypergeometric series defined by
\begin{equation}\label{hypergeom}
\hypergeomq{r}{s}{(a_r)}{(b_s)}{q}{z}=
\sum_{n=0}^\infty{[a_r;q]_n\over [b_s;q]_n}\left((-1)^nq^{{n(n-1)\over
2}}\right)^{1+s-r}{z^n\over (q;q)_n},
\end{equation}
where
$$[a_n;q]_k=(a_1;q)_k(a_2;q)_k\cdots(a_n;q)_k,\quad k=0,1,2,\ldots.$$
The base $q$ will be restricted to $|q| < 1$ for non-terminating series.\\
The contracted notation $(a_r)$ is used to
abbreviate the array of $r$ parameters $a_1,\cdots a_r$,\ldots.\\
The parameters $a, (a_r)$ and $(b_s)$ are assumed to be independent of $n$.\\
The closed analytical formulae for the corresponding inversion and connection coefficients will be expressed by means of terminating basic hypergeometric functions which, in some cases, can be evaluated as a basic hypergeometric terms.\\
By applying appropriate limit to the obtained results we derive connection coefficients for the following hypergeometric polynomials:
\begin{eqnarray}
\label{braf}
\ _{r+1}F_s\left(\begin{array}{l}-n,\ (a_r)\\
(b_s)\end{array};x\right),
\end{eqnarray}
\begin{eqnarray}
\label{srpa}
 \ _{r+2}F_{s}\left(\begin{array}{l}-n,\lambda+n,(a_r)\\
(b_s)\end{array};x\right),
 \end{eqnarray}
 where the  $_rF_s$ denotes, as usual, the
  generalized hypergeometric
 functions with $r$ numerator and $s$ denominator
 parameters. \\
 The polynomials defined by (\ref{basic2})-(\ref{srpa}) are relevant to the study
   of quantum-mechanical systems and include as particular cases many known polynomial sets,
 we quote for instance orthogonal polynomials in Askey scheme and its $q$-analogue and their generalizations with
     Sobolev type orthogonality and $d$-orthogonality.\\
The families (\ref{basic1}) and (\ref{hypergeom}) extend the hypergeometric polynomials given by (\ref{braf}) and (\ref{srpa}). It is clear that
$$\lim_{q\rightarrow 1^-}\frac{(q^a;q)_k}{(1-q)^k}=(a)_k,$$
hence
\begin{eqnarray}\label{lim qhyp hyp}
\lim_{q\mapsto1^-}\hypergeomq{r}{s}{q^{a_1},\ldots,q^{a_r}}{q^{b_1},\ldots,q^{b_s}}{q}{(1-q)^{s+1-r}x}=\hypergeom{r}{s}{(a_r)}{(b_s)}{x}.
\end{eqnarray}
We apply the obtained results to some $d$-orthogonal basic hypergeometric polynomials (Big $q$-Laguerre type, Little $q$-Laguerre type, $q$-Laguerre type and $q$-Meixner type) as well as to the orthogonal basic hypergeometric polynomials of the $q$-analogue of Askey scheme (Askey-Wilson, $q$-Racah, Continuous dual $q$-Hahn, Continuous $q$-Hahn, Big $q$-Jacobi,... ).\\
The content of this paper is organized as follows:
\begin{enumerate}
\item In Section~\ref{connectionbasic}, we prove our main
result~Theorem~\ref{thmmain} as well as some useful consequences.
\item In
Section~\ref{connectiondortho}, we apply our results to many
generalized
 basic hypergeometric polynomial sets studied in the framework of $d$-orthogonality.
\item In Section \ref{connectionqpoly}, the explicit inversion and connection
coefficients between orthogonal polynomials of the $q$-Askey scheme are summarized in Tables~\ref{inversionq-askey}-\ref{connectionq-askey}.
\end{enumerate}
\section{Connection coefficients between basic Hypergeometric polynomials}
\label{connectionbasic}
\begin{thm}\label{thmmain}
The inversion and connection formulae for basic hypergeometric polynomials (\ref{basic1}) are given by \\
\begin{equation}\label{inversiongen}
x^n=\frac{[b_r;q]_n}{[a_s;q]_n}\left((-1)^nq^{\frac{n(n-1)}{2}}\right)^{r+1-s}\mathop\sum_{m=0}^n
{n\brack m}_q\frac{(-1)^mq^{\frac{m(m-1)}{2}}}{(aq^m;q)_m(aq^{2m+1};q)_{n-m}}\hypergeomq{r+2}{s}{q^{-m},aq^m,(a_r)}{(b_s)}{q}{qx},
\end{equation}
and
\begin{eqnarray}\label{connectiongen}
  \hypergeomq{r+2}{s}{q^{-n},aq^n,(a_r)}{(b_s)}{q}{qx} &=& \mathop\sum_{m=0}^n{n\brack m}_q(-1)^{m(s+l-r-h)}
\frac{(aq^n;q)_m}{(cq^m,q)_m}\frac{[a_r;q]_m}{[b_s;q]_m}
\frac{[d_h;q]_m}{[c_{l};q]_m} \nonumber\\
   &\times& \hypergeomq{r+h+2}{s+l+1}{q^{m-n},aq^{m+n},(a_rq^m),(d_hq^m)}{cq^{2m+1},(b_sq^m),(c_{l}q^m)}{q}{q^{1+m(s+l-r-h)}}\nonumber \\
   &\times& q^{\frac{m(m-1)}{2}(s+l-r-h)}q^{m(m-n)}\hypergeomq{l+2}{h}{q^{-m},cq^m,(c_l)}{(d_h)}{q}{qx},
\end{eqnarray}
where, the $q$-binomial coefficient ${n\brack k}_q$ is defined by ${n\brack k}_q={(q;q)_n\over (q;q)_k(q;q)_{n-k}}.$
\end{thm}
Starting point for the proof of this theorem is the following result.
 \begin{lem}\cite[Theorem 2.1]{neila2015a}
 Let $\{P_n\}_{n\geq0}$ be a monic polynomial set expanded in a given basis $\{B_n\}_{n\geq0}$ by
 \begin{equation}
 P_n(x)=\displaystyle\sum_{k=0}^n A_k(n)B_{n-k}(x).
\end{equation}
Then the following inversion formula holds
\begin{equation}\label{inversionlem}
B_n(x)=\displaystyle\sum_{m=0}^n b_m(n,0)P_{n-m}(x),
\end{equation}
where
\begin{equation}\label{recurrence}
\left\{\begin{array}{lll}
         b_0(n,0)=1, &  &  \\
         b_0(n,k)=0 & \textrm{if} & 1\leq k, \\
         b_{m+1}(n,k)=b_m(n,k+1)-b_m(n,0)A(n-m,k+1) & \textrm{if} & 0\leq m\leq n-1\;,0\leq k\leq n-m-1.
       \end{array}
\right.
\end{equation}
 \end{lem}
\begin{proof}
In order to get the explicit expression of the inversion coefficient $b_m(n, 0)$ in (\ref{inversionlem}), we need to compute the coefficients $b_m(n, k)$ for each
$m$ and $k$. For this, we use the recurrence relation (\ref{recurrence}) to compute some initial
 values and to guess the resulting term. Then we proceed by induction to give the proof.\\
Let $\tilde{P}_n(x)$ be the monic basic hypergeometric polynomial defined by
\begin{eqnarray*}
\tilde{P}_n(x)=\frac{(-1)^n[b_s;q]_n}{(aq^n;q)_n[a_r;q]_n}\left[(-1)^nq^{\frac{n(n-1)}{2}}\right]^{r+1-s}q^{-\frac{n(n-1)}{2}}
\hypergeomq{r+2}{s}{q^{-n},aq^n,(a_r)}{(b_s)}{q}{qx},
\end{eqnarray*}
which by some elementary transformations can also be written as
\begin{equation*}
\tilde{P}_n(x)=\displaystyle\sum_{k=0}^n(-1)^k{n\brack k}_q\left[(-1)^kq^{-\frac{k(2n-k-1)}{2}}\right]^{s-r-1}
q^{\frac{k(k-1)}{2}}\frac{(aq^{n-k};q)_{n-k}}{(aq^n;q)_n}\frac{[b_sq^{n-k};q]_k}{[a_rq^{n-k};q]_k} x^{n-k}.
\end{equation*}
It follows
 \begin{eqnarray*}
 b_1(n,k)&=&-A(n,k+1)\\
 &=&(-1)^k{n\brack k+1}_q\left[(-1)^{k+1}q^{-\frac{(k+1)(2n-k-2)}{2}}\right]^{s-r-1}q^{\frac{k(k+1)}{2}}\\
&\times& \frac{(aq^{n-k-1};q)_{n-k-1}}{(aq^n;q)_n}\frac{[b_sq^{n-k-1};q]_{k+1}}{[a_rq^{n-k-1};q]_{k+1}}.
\end{eqnarray*}
For $b_2(n,k)$, we have
\begin{eqnarray*}
b_1(n,0)A(n-1,k+1)&=&(-1)^{k+1}\left[\begin{array}{c}n\\k+2\end{array}\right]_q\frac{(q^{k+2};q)_1}{(q;q)_1}\left[(-1)^kq^{-\frac{(k+2)(2n-k-3)}{2}}\right]^{s-r-1}
q^{\frac{k(k+1)}{2}}\\
&\times&\frac{[b_sq^{n-k-2};q]_{k+2}}{[a_rq^{n-k-2};q]_{k+2}}\frac{(aq^{n-k-2};q)_{n-k-2}}{(aq^n;q)_n}.
\end{eqnarray*}
By the useful identities,
$${n\brack m}_q{n-m\brack k}_q=
{n\brack k+m}_q\frac{(q^{k+1};q)_m}{(q;q)_m},\ (aq;q)_n=(a;q)_n\frac{1-aq^n}{1-a}\ \textrm{and}\ (a;q)_{n+m}=(a;q)_m(aq^m;q)_{n-m},$$
we find
\begin{eqnarray*}
b_2(n,k)&=&(-1)^k\left[\begin{array}{c}n\\k+2\end{array}\right]_q
\left[(-1)^{k+2}q^{-\frac{(k+2)(2n-k-3)}{2}}\right]^{s-r-1}q^{\frac{k(k+1)}{2}}\\
&\times&\frac{(aq^{n-k-2};q)_{n-k-2}}{(aq^n;q)_n}\frac{[b_sq^{n-k-2};q]_{k+2}}{[a_rq^{n-k-2};q]_{k+2}}
\frac{(q^{k+1};q)_{1}}{(q;q)_{1}}.
\end{eqnarray*}
Now, we can suggest the following form of $b_m(n,k)$:
\begin{eqnarray*}
b_m(n,k)&=&(-1)^k{n\brack k+m}_q
\left[(-1)^{k+m}q^{-\frac{(k+m)(2n-k-m-1)}{2}}\right]^{s-r-1}
q^{\frac{k(k+1)}{2}}\nonumber\\
&\times&\frac{(aq^{n-k-m};q)_{n-k-m}}{(aq^n;q)_n}\frac{[b_sq^{n-k-m};q]_{k+m}}
{[a_rq^{n-k-m};q]_{k+m}}\frac{(q^{k+1};q)_{m-1}}{(q;q)_{m-1}}.
\end{eqnarray*}
Assuming the formula to hold for $m$. Substituting $b_m(n,0), b_m(n,k+1)$ and $A(n-m,k+1)$ with their expressions in $b_{m+1}(n,k)$ given by (\ref{recurrence}), it follows that the assumption is valid for $m+1$. Thus, up taking $k = 0$, we obtain:
\begin{equation*}\label{inv}
b_m(n,0)={n\brack m}_q\left[(-1)^mq^{-\frac{m(2n-m-1)}{2}}\right]^{s-r-1}
\frac{(aq^{n-m};q)_{n-m}}{(aq^n;q)_n} \frac{[b_sq^{n-m};q]_m}{[a_rq^{n-m};q]_m}.
\end{equation*}
That leads to (\ref{inversiongen}).\\
According to the basic hypergeometric representation given by (\ref{basic1}), to the associated inversion formula (\ref{inversiongen}) and to the composition formula (\ref{formule connection}) and by the relation $$(q^{-n};q)_m=\frac{(-1)^m(q;q)_n}{(q;q)_{n-m}}q^{\frac{m(m-1)}{2}}q^{-nm}$$
we get
\begin{eqnarray*}
C_m(n)&=&(-1)^m\frac{(q^{-n};q)_m(aq^n;q)_m[a_r;q]_m}{[b_s;q]_m}q^m
\frac{[d_h;q]_m}{[c_{l};q]_m(cq^m;q)_m}q^{\frac{m(m-1)}{2}}\\
&\times&\mathop\sum_{j=0}^{n-m}{
j+m\brack m}_q\frac{(q^{m-n};q)_j(aq^{n+m};q)_j[a_rq^m;q]_j}{(q;q)_{m+j}[b_sq^m;q]_j}\\
&\times& \frac{[d_hq^m;q]_j}{(cq^{2m+1};q)_j[c_{l}q^m;q]_j}q^j \left[(-1)^{m+j}q^{\frac{(m+j)(m+j-1)}{2}}\right]^{s-r+l-h}\\
&=&q^{m(m-n)}
\frac{(q;q)_n}{(q;q)_{n-m}(q;q)_m}
\frac{(aq^n;q)_m[a_r;q]_m[d_h;q]_m}{[b_s;q]_m[c_{l};q]_m(cq^m;q)_m}\\
&\times&\mathop\sum_{j=0}^{n-m}\frac{(q^{m-n};q)_j(aq^{m+n};q)_j[a_rq^m;q]_j[d_hq^m;q]_j}
{(cq^{2m+1};q)_j[b_sq^m;q]_j[c_{l}q^m;q]_j(q;q)_j}\\
&\times&(-1)^{m(s+l-r-h)}\left[q^{\frac{m(m-1)}{2}}\right]^{s+l-r-h}
\left[(-1)^{j}q^{\frac{j(j-1)}{2}}\right]^{s+l-r-h}q^j\left(q^{mj}\right)^{s+l-r-h}.
\end{eqnarray*}
Then (\ref{connectiongen}) follows.
\end{proof}
If we put $a=c=0$ in (\ref{inversiongen}) and (\ref{connectiongen}) we obtain the inversion and connection formula for basic hypergeometric polynomials (\ref{basic2}).
\begin{cor}
The inversion and connection formulae associated to (\ref{basic2}) are given by:
\begin{equation}\label{inversionbasic1}
x^n=\frac{[b_s;q]_n}{[a_r;q]_n}\left((-1)^nq^{\frac{n(n-1)}{2}}\right)^{r-s}\mathop\sum_{k=0}^n(-1)^k
{n\brack k}_qq^{\frac{k(k-1)}{2}}\hypergeomq{r+1}{s}{q^{-k},(a_r)}{(b_s)}{q}{qx},
\end{equation}
and
\begin{eqnarray}\label{connectionbasic1}
  \hypergeomq{r+1}{s}{q^{-n},(a_r)}{(b_s)}{q}{qx} &=& \mathop\sum_{m=0}^n{n\brack m}_q
(-1)^{m(s+l-r-h)}q^{\frac{m(m-1)}{2}(s+l-r-1-h)}q^{m(m-n)} \nonumber\\
   &\times& \frac{[a_{r};q]_m}{[b_s;q]_m}\frac{[d_h;q]_m}{[c_{l};q]_m}
\hypergeomq{r+h+1}{s+l}{q^{m-n},(a_rq^m),(d_hq^m)}{(b_sq^m),(c_{l}q^m)}{q}{q^{1+m(s+l-r-h)}}\nonumber \\
   &\times& \hypergeomq{l+1}{h}{q^{-m},(c_l)}{(d_h)}{q}{qx}.
\end{eqnarray}
\end{cor}
The following two corollaries provide inversion and connection formulae for generalized hypergeometric polynomials given by (\ref{srpa}) and (\ref{braf}). These results can be obtained as a limit cases of Eqs. (\ref{inversiongen}), (\ref{connectiongen}), (\ref{inversionbasic1}) and (\ref{connectionbasic1}) when $q\rightarrow 1^-$ and by using (\ref{lim qhyp hyp}).
\begin{cor} The inversion and connection formulae associated to (\ref{srpa}) are given by
\begin{equation}\label{invsrpa}
x^n=\frac{(b_s)_n}{(a_r)_n}\mathop\sum_{m=0}^n\left(\begin{array}{c}n\\m\end{array}\right)
\frac{(-1)^m}{(\lambda+m)_m(\lambda+2m+1)_{n-m}}
\hypergeom{r+2}{s}{-m,\lambda+m,(a_r)}{(b_s)}{x},
\end{equation}
and
\begin{eqnarray}\label{connectionsrpa}
\hypergeom{r+2}{s}{-n,\lambda+n,(a_r)}{(b_s)}{x}&=&\mathop\sum_{m=0}^n\left(\begin{array}{c}n\\m\end{array}\right)
\frac{(\lambda+n)_m}{(\beta+m)_m}\frac{[a_r]_m}{[b_r]_m}\frac{[d_s]_m}{[c_r]_m}\nonumber\\
&\times&\hypergeom{r+s+2}{r+s+1}{m-n,\lambda+n+m,(a_r+m),(d_s+m)}{\beta+2m+1,(b_s+m),(c_r+m)}{1}\nonumber\\
&\times&\hypergeom{r+2}{s}{-m,\beta+m,(c_r)}{(d_s)}{x}.
\end{eqnarray}
\end{cor}
\begin{cor} The inversion and connection formulae associated to (\ref{braf}) are given by
\begin{equation}\label{inversionbraf}
x^n=\frac{[b_s]_n}{[a_r]_n}\mathop\sum_{m=0}^n\left(\begin{array}{c}n\\m\end{array}\right)(-1)^m
\hypergeom{r+1}{s}{-m,(a_r)}{(b_s)}{x},
\end{equation}
and
\begin{eqnarray}\label{connectionbraf}
\hypergeom{r+1}{s}{-n,(a_r)}{(b_s)}{x}&=&\mathop\sum_{m=0}^n\left(\begin{array}{c}n\\m\end{array}\right)\frac{[a_r]_m}{[b_r]_m}\frac{[d_s]_m}{[c_r]_m}\nonumber\\
&\times&\hypergeom{r+s+1}{r+s}{m-n,(a_r+m),(d_s+m)}{(b_s+m),(c_r+m)}{1}\nonumber\\
&\times&\hypergeom{r+1}{s}{-m,(c_r)}{(d_s)}{x}.
\end{eqnarray}
\end{cor}
The inversion and connection formulae (\ref{invsrpa})--(\ref{connectionbraf}) were already given in \cite{chaggara2006a} via generating functions manipulations.

Note here that, as in the classical Laguerre case \cite{rota1973}, the generalized hypergeometric polynomial:
$$\tilde{L}_n^{((\alpha_r);(\beta_s))}(x):
=\frac{[\beta_s]_n}{[\alpha_r]_n}
\ _{r+1}F_s\left(\begin{array}{l}-n,\ (\alpha_r)\\
(\beta_s)\end{array};x\right)$$
 is self inverse. That is to say, polynomial with same expression of the coefficients in the direct and in the inverse formulae. This can be translated in the \textit{Umbral calculus} context by:
$$\tilde{L}_n^{((\alpha_r);(\beta_s))}
(\tilde{\textbf{L}}_n^{((\alpha_r);(\beta_s))}(x))=x^n.$$
\section{Connection coefficients between $d$-orthogonal basic polynomials}
 \label{connectiondortho}
Next we apply the obtained results to some $d$-orthogonal basic hypergeometric polynomials.\\
The notion of $d$-orthogonal polynomials generalize the standard orthogonal polynomials in that they satisfy a $d$ orthogonality conditions and they obey a higher-order recurrence relation \cite{maroni1989}.
This kind of orthogonality appears as a special case of the general multiple orthogonality. In fact, the $d$-orthogonal polynomials correspond to multiple orthogonal polynomials near the diagonal \cite{ismail2005}.\\
The concept of $d$-orthogonality has been the subject of numerous investigations and applications. In particular, it is connected with the study of vector Pad\'{e} approximants, vectorial continued fractions, resolution of higher-order differential equations and spectral study of multi-diagonal nonsymmetric operators.\\
Most of the known explicit examples of $d$-orthogonal polynomial sets were introduced by solving a characterization
problem that consists in finding all $d$-orthogonal polynomials, satisfying a given property.\\
The basic hypergeometric $d$-orthogonal polynomials that will be considered here generalize, firstly, the known $q$-Meixner, big $q$-Laguerre, little $q$-Laguerre and $q$-Laguerre orthogonal polynomials and, on the other hand, can be viewed as $q$-analogs of the $d$-orthogonal polynomials of Meixner and Laguerre type.\\
The inversion coefficients for $d$-orthogonal polynomials were used to derive the $d$-dimensional functional vectors ensuring their $d$-orthogonality \cite{neila2015b}.
The $q$ analogs of the $d$-orthogonal polynomial sets were introduced and investigated in details in \cite{lamiri2013}.\\
\paragraph{$d$-orthogonal little $q$-Laguerre type}\ \\
The $d$-orthogonal of little $q$-Laguerre type polynomials are defined by the basic hypergeometric sum \cite{lamiri2013}:
\begin{equation}
p_n(x,(b_s)/q)=\hypergeomq{d+1}{s}{q^{-n},0,\dots,0}{(b_s)}{q}{qx}.
\end{equation}
By (\ref{inversionbasic1}) and (\ref{connectionbasic1}), the inversion and connection formulae are given by
\begin{equation}
x^n=\mathop\sum_{m=0}^n{n\brack m}_q
\left((-1)^nq^{\frac{n(n-1)}{2}}\right)^{d-s}[b_s;q]_n(-1)^mq^{\frac{m(m-1)}{2}}p_n(x,(b_s)/q),
\end{equation}
and
\begin{equation}
p_n(x,(b_s)/q) =\mathop\sum_{m=0}^n{n\brack m}_qq^{m(m-n)}\frac{[\beta_s;q]_m}{[b_s;q]_m}
\hypergeomq{s+1}{s}{q^{m-n},(\beta_sq^m)}{(b_sq^m)}{q}{q} p_m(x,(\beta_s)/q).
\end{equation}
\paragraph{$d$-orthogonal $q$-Meixner type}\ \\
The $d$-orthogonal polynomials of $q$-Meixner type are defined by \cite{lamiri2013}
\begin{equation}
M_n(q^{-x};(b_d),c;q)=\hypergeomq{2}{d}{q^{-n},q^{-x}}{(d_d)}{q}{-\frac{q^{n+1}}{c}}
\end{equation}
The following inversion and connection relations are valid:
\begin{equation}
(q^{-x};q)_n=\mathop\sum_{m=0}^n [b_d;q]_n\left((-1)^nq^{\frac{n(n-1)}{2}}\right)^{1-d}(-1)^{m+n}q^{\frac{m(m-1)}{2}-mn}(c)^n
{n\brack m}_qM_m(q^{-x};(b_d),c;q),
\end{equation}
and
\begin{eqnarray}
M_n(q^{-x};(b_d),c;q)&=&\mathop\sum_{m=0}^n \left(\frac{c}{\gamma}\right)^m{n\brack m}_q\frac{[\beta_d;q]_m}{[b_d;q]_m}\nonumber\\
&\times &
\hypergeomq{d+1}{d}{q^{m-n},(\beta_dq^m)}{(b_d)}{q}{\frac{\gamma}{c}q^{1+n-m}}M_m(q^{-x};(\beta_d),\gamma;q).
\end{eqnarray}
\paragraph{$d$-orthogonal big $q$-Laguerre type}\ \\
The $d$-orthogonal polynomial of big $q$-Laguerre type has the following $q$-hypergeometric representation \cite{lamiri2013}:
\begin{equation}
P_n(x;(b_{d+1});q)=\hypergeomq{d+2}{d+1}{q^{-n},0,\dots,0,x}{(b_{d+1})}{q}{q}.
\end{equation}
This polynomial set fulfil the following inversion ans connection formulae:
\begin{equation}
(x;q)_n=\mathop\sum_{m=0}^n[b_{d+1};q]_nq^{\frac{m(m-1)}{2}}
{n\brack m}_qP_m(x;(b_{d+1});q),
\end{equation}
and
\begin{eqnarray}
P_n(x;(b_{d+1});q)&=&\mathop\sum_{m=0}^n {n\brack m}_qq^{m(m-n)}\frac{[\beta_{d+1};q]_m}{[b_{d+1};q]_m}\nonumber\\
&\times &\hypergeomq{d+2}{d+1}{q^{m-n},(\beta_{d+1}q^m)}{(b_{d+1}q^m)}{q}{q} P_m(x;(\beta_{d+1});q)
\end{eqnarray}
\paragraph{$d$-orthogonal $q$-Laguerre type}\ \\
The explicit expression of the $d$-orthogonal polynomial of $q$-Laguerre type is as follows \cite{lamiri2013}:
\begin{equation}
L_n^{(b_1,\dots,b_d)}(x;q)=\hypergeomq{1}{d}{q^{-n}}{(b_d)}{q}{q^nx}
\end{equation}
The $d$-orthogonal $q$-Laguerre polynomials satisfy the following inversion and connection formulae:
\begin{equation}\label{inqlag}
x^n=[b_d;q]_n\mathop\sum_{m=0}^n {n\brack m}_q
\left((-1)^nq^{\frac{n(n-1)}{2}}\right)^{-d}q^{\frac{m(m-1)}{2}-mn}(-1)^mL_m^{(b_1,\dots,b_d)}(x;q),
\end{equation}
and
\begin{eqnarray}\label{connqlag}
L_n^{(b_1,\dots,b_d)}(x;q)&=&\mathop\sum_{m=0}^n{n\brack m}_q \frac{[\beta_{d};q]_m}{[b_{d};q]_m}\hypergeomq{d+1}{d}{q^{m-n},(\beta_{d}q^m)}{(b_{d}q^m)}{q}{q^{1+n-m}}L_m^{(\beta_1,\dots,\beta_d)}(x;q).
\end{eqnarray}
\section{Connection relation for basic orthogonal hypergeometric polynomials}
\label{connectionqpoly}
In this section we deal with families of basic hypergeometric orthogonal polynomials
appearing in the $q$-Askey scheme. We give the corresponding inversion and connection coefficients without using the orthogonality property. In this way, the obtained formulae are still valid outside the range of orthogonality of the parameters.\\
 The $q$-Askey scheme is a $q$-analogue of the Askey scheme. It contains specific orthogonal polynomials which can be written in terms of basic hypergeometric functions starting in the top with Askey-Wilson polynomials and $q$-Racah polynomials and ending in the bottom with continuous and discrete $q$-Hermite polynomials and Stieltjes-Wigert polynomials. The inversion and connection formulae for Askey-Wilson and $q$-Racah polynomials follow from (\ref{inversiongen}) and (\ref{connectiongen}) and they will be given in the sequel of this section. The corresponding expansions for the remainder polynomials within the $q$-Askey scheme are quoted in Tables~\ref{inversionq-askey} and~\ref{connectionq-askey}. They can be obtained either as limit cases and specialization process from Askey-Wilson and $q$-Racah polynomials or by mean of the general inversion and connection formulae obtained in Section \ref{connectionbasic}. The polynomial basis considered in Table~\ref{inversionq-askey} are suggested by the basic hypergeometric representation of each family (see \cite{koekoek2010}, for more details).
\paragraph{Askey-Wilson}\ \\
The Askey–Wilson polynomial set is a family of orthogonal polynomials introduced by Askey and Wilson in \cite{askeywilson85} as $q$-analogs of the Wilson polynomials. They include many of other orthogonal polynomials as special or limiting cases. The Askey-Wilson polynomials belong to the so-called classical orthogonal polynomials on a non uniform lattice which are known to satisfy a particular divided-difference equation \cite[Chapter14]{koekoek2010}. The Askey-Wilson polynomials are defined by \cite{askeywilson85}
 \begin{equation}\label{inaskwil}
P_n(x;a,b,c,d/q)=\frac{(ab,ac,dq;q)_n}{a^{n}}
\hypergeomq{4}{3}{q^{-n},abcdq^{n-1},ae^{i\theta},ae^{-i\theta}}{ab,ac,ad}{q}{q},\;x=\cos(\theta),
 \end{equation}
where $(ab,ac,dq;q)_n:=(ab;q)_n(ac;q)_n(dq;q)_n$.

The following expansion formulae of Askey-Wilson basis in terms of Askey-Wilson polynomials is valid:
\begin{equation}\label{conaskwil}
   (ae^{i\theta};q)_n (a^{-i\theta};q)_n=\mathop\sum_{m=0}^n{n\brack m}_qq^{\frac{m(m-1)}{2}}(-a)^m\frac{(abq^m,acq^m,adq^m;q)_{n-m}}
   {(abcdq^{m-1};q)_m(abcdq^{2m};q)_{n-m}}P_m(x;a,b,c,d/q).
   \end{equation}
The connection formula between two Askey-Wilson polynomials, with first common parameter, is given by
   \begin{eqnarray}
 P_n(x,a,b,c,d/q)&=&\mathop\sum_{m=0}^n\frac{a^{m-n}q^{m(m-n)}}{(q;q)_{n-m}}
 \frac{(ab,ac,ad,q;q)_n(abcdq^{n-1};q)_m}{(ab,ac,ad,q;q)_m(a \beta\gamma\delta q^{m-1};q)_m}\nonumber \\
&\times&\hypergeomq{5}{4}{q^{m-n},a\beta q^m,a\gamma q^m,a\delta q^m,abcdq^{n+m-1}}{abq^m,acq^m,adq^m,a\beta\gamma\delta q^{2m}}{q}{q}
P_m(x,a,\beta,\gamma,\delta/q).
 \end{eqnarray}
Note that, the inversion and connection problems for Askey-Wilson polynomials were already studied by many authors: Askey and Wilson used orthogonality assumption \cite{askeywilson85}, Area et al. used Verma Formula \cite{navima2001}. However, Foupouagnigni et al. solved this problem recurrently by computer algebra tools \cite{mama2013}.
\paragraph{$q$-Racah}\ \\
The $q$-Racah polynomials is a set of orthogonal polynomials that generalize the Racah coefficients or $6-j$ symbols.
They were introduced in \cite{askeywilson79}. Their hypergeometric representation is given by
\begin{equation}
 R_n(\nu(x);\alpha,\beta,\delta;\gamma/q)=\hypergeomq{4}{3}{q^{-n},\alpha\beta q^{n+1},q^{-x},\delta\gamma q^{x+1}}{\alpha q,\beta \delta q,\gamma q}{q}{q},\;n=0,1,...,N,
 \end{equation}
 where $$\nu(x)=q^{-x}+\delta\gamma q^{x+1},$$
 and $\alpha q= q^{-N}$ for some integer $N$.\\
 Clearly, $R_n(\nu(x);\alpha,\beta,\delta;\gamma/q)$ is a polynomial of degree $n$ in $\nu(x)$ and the case $q\rightarrow 1$ gives the Racah polynomials.

By (\ref{inversiongen})-(\ref{connectiongen}), the following inversion and connection formulae of $q$-Racah polynomials are valid:
 \begin{equation}
 (q^{-x},\gamma\delta q^{x+1};q)_n=\mathop\sum_{m=0}^n(-1)^m{ n\brack m }_qq^{\frac{m(m-1)}{2}}\frac{(\alpha q,\beta\delta q,\gamma q;q)_n}{(\alpha\beta q^{m+1};q)_m(\alpha\beta q^{2m+2};q)_{n-m}}R_m(\nu(x);\alpha,\beta,\delta;\gamma/q),
 \end{equation}
 and
 \begin{eqnarray}
  R_n(\mu(x);\alpha,\beta,\gamma,\delta/q)&=&\mathop\sum_{m=0}^n { n\brack m }_qq^{m(m-n)}\frac{(\alpha\beta q^{n+1},aq,bd q,cq;q)_m}{(\alpha q,\beta\delta q,\gamma q,abq^{m+1};q)_m}\nonumber\\
 &\times&\hypergeomq{5}{4}{q^{m-n},\alpha\beta q^{m+n+1},aq^{m+1},bd q^{m+1},cq^{m+1}}{\alpha q^{m+1},\beta\delta q^{m+1},abq^{2m+2},\gamma q^{m+1}}{q}{q}
 R_m(\mu(x);a,b,c,d/q),\nonumber\\
 \end{eqnarray}
  provided that $\gamma\delta=cd$.
\begin{center}
\footnotesize
{\renewcommand{\arraystretch}{2} 
{\setlength{\tabcolsep}{0.1cm} 
\begin{longtable}{|c c c|}\caption{{Inversion coefficients in the $q$-Askey scheme}} \label{inversionq-askey}\\
\hline
$\textbf{Polynomial set}$ &\begin{tabular}{cc}  $\textbf{Polynomial sets  basis}$\\ $\textbf{$\{B_n\}_n\geq0$}$\end{tabular}& $\textbf{Inversion coefficients $I_m(n)$}$ \\
\hline
\endfirsthead
Askey-Wilson:
$P_n(x;a,b,c,d/q)$  & $(ae^{i\theta},ae^{-i\theta};q)_n$ &${n\brack m}_qq^{\frac{m(m-1)}{2}}(-a)^m\frac{(abq^m,acq^m,adq^m;q)_{n-m}}{(abcdq^{m-1};q)_m(abcdq^{2m};q)_{n-m}}$ \\
\hline
$q$-Racah $: R_n(\nu(x);\alpha,\beta,\delta;\gamma/q)$  &$(q^{-x},\gamma\delta q^{x+1};q)_n$ &$(-1)^m{n\brack m}_qq^{\frac{m(m-1)}{2}}\frac{(\alpha q,\beta\delta q,\gamma q;q)_n}{(\alpha\beta q^{m+1};q)_m(\alpha\beta q^{2m+2};q)_{n-m}}$ \\
\hline
 Big $q$-Jacobi $: P_n(x;a,b,c;q)$ & $(x;q)_n$ &  $(-1)^m{n\brack m}_qq^{\frac{m(m-1)}{2}}\frac{(aq,cq;q)_n}{(abq^{2m+2};q)_{n-m}(abq^{m+1};q)_m}$\\
\hline
 $q$-Hahn $: Q_n(q^{-x};\alpha,\beta,N/q)$  &  $(q^{-x};q)_n$ & ${n\brack m}_qq^{\frac{m(m-1)}{2}}\frac{(q^{-N},\alpha q;q)_n(-1)^m}{(\alpha\beta q^{m+1};q)_m(\alpha\beta q^{2m+2};q)_{n-m}}$ \\
\hline
 Dual $q$-Hahn $: R_n(\nu(x);\gamma,\delta,N/q)$  & $(q^{-x},\gamma\delta q^{x+1};q)_n $ & $(\gamma q,q^{-N};q)_n(-1)^m{n\brack m}_qq^{\frac{m(m-1)}{2}} $ \\
\hline
 Al-Salam-Chihara $: Q_n(x;a,b/q)$&  $(ae^{i\theta},ae^{-i\theta};q)_n $ & $(-a)^m(abq^m;q)_{n-m}{n\brack m}_qq^{\frac{m(m-1)}{2}}$ \\
\hline
 $q$-Meixner-Pollaczek $: P_n(x;a/q)$& $ (ae^{i(\theta+2\phi)},ae^{-i\theta};q)_n$ & $(a^2q^m;q)_{n-m}(q;q)_m(-ae^{i\phi})^m {n\brack m}_qq^{\frac{m(m-1)}{2}}$ \\
\hline
  \multirow{2}*{Continuous $q$ Jacobi $: P_n^{(\alpha,\beta)}(x/q)$} &   $\left(q^{\frac{1}{2}\alpha+\frac{1}{4}}e^{i\theta};q\right)_n$ & $ (-q^{\frac{1}{2}(\alpha+\beta+1)},-q^{-\frac{1}{2}(\alpha+\beta+2)};q)_n{n\brack m}_q$\\
  &$\times\left(q^{\frac{1}{2}\alpha+\frac{1}{4}}e^{-i\theta};q\right)_n $& $\times \frac{(-1)^m(q;q)_m(q^{\alpha+m+1};q)_{n-m}}{(q^{m+\alpha+\beta+1};q)_m(q^{2m+\alpha+\beta+2};q)_{n-m}}q^{\frac{m(m-1)}{2}}$ \\
\hline
\begin{tabular}{cc}Continuous $q$-Ultraspherical/ \\
Rogers: $C_n(x;\beta/q)$ \end{tabular}&  $\left(\beta^{\frac{1}{2}}e^{i\theta},\beta q^{-i\theta};q\right)_n$ & \begin{tabular}{cc}${n\brack m}_qq^{\frac{m(m-1)}{2}}\frac{\left(\beta q^{\frac{1}{2}},-\beta,-\beta q^{\frac{1}{2}};q\right)_n}{(\beta^2q^m;q)_m(\beta^2;q)_m}$\\$\times\frac{\left(-\beta^{\frac{1}{2}}\right)^m(q;q)_m}{(\beta^2q^{2m+1};q)_{n-m}} $ \end{tabular}\\
\hline
Continuous $q$-Legendre: $P_n(x/q)$ & $\left(q^{\frac{1}{4}}e^{i\theta},q^{\frac{1}{4}}e^{-i\theta};q\right)_n $ & ${n\brack m}_q\frac{\left(q,-q^{\frac{1}{2}},-q;q\right)_n(-1)^m}{(q^{m+1};q)_m(q^{2m+2};q)_{n-m}}q^{\frac{m(m-1)}{2}} $ \\
\hline
Big $q$-Laguerre: $P_n(x;a,b,q)$  &$(x;q)_n $&$(aq,bq;q)_n(-1)^m{n\brack m}_qq^{\frac{m(m-1)}{2}} $\\
\hline
Little $q$-Jacobi: $P_n(x;a,b/q) $ &$ x^n$&${n\brack m}_q\frac{(-1)^m(aq;q)_n}{(abq^{m+1};q)_m(abq^{2m+2};q)_{n-m}} q^{\frac{m(m-1)}{2}}$\\
\hline
Little $q$-Legendre: $P_n(x/q) $ &$x^n $&$ {n\brack m}_q\frac{q^{\frac{m(m-1)}{2}}(-1)^m(q;q)_n}{(q^{m+1};q)_m(q^{2m+2};q)_{n-m}}$\\
\hline
$q$-Meixner: $M_n(q^{-x},b,c;q) $ &$(q^{-x};q)_n $&$(bq;q)_n(-1)^{m+n}c^n{n\brack m}_qq^{\frac{m(m-1)}{2}-mn} $\\
\hline

\begin{tabular}{cc}Quantum $q$-Krawtchouk \\$K_n(q^{-x};p,-N;q)$\end{tabular}& $(q^{-x};q)_n $&$\left(p\right)^{-n}(q^{-N};q)_n(-1)^m{n\brack m}_qq^{\frac{m(m-1)}{2}-mn} $\\
\hline
$q$-Krawtchouk: $K_n(q^{-x};p,N;q)$ &$(q^{-x};q)_n$ & $(-1)^m{n\brack m}_qq^{\frac{m(m-1)}{2}}\frac{(q^{-N};q)_n}{(-pq^m;q)_m(-pq^{2m+1};q)_{n-m}} $\\
\hline
\begin{tabular}{cc}Affine $q$-Krawtchouk: \\$K_n(q^{-x};p,N;q)$ \end{tabular} & $(q^{-x};q)_n$ & $(pq,q^{-N};q)_n(-1)^m{n\brack m}_qq^{\frac{m(m-1)}{2}}$\\
\hline
\begin{tabular}{cc}Dual $q$-Krawtchouk: \\ $K_n(\lambda(x);c;N/q)$\end{tabular} &$(q^{-x},cq^{x-N};q)_n$& $(q^{-N};q)_n(-1)^m{n\brack m}_qq^{\frac{m(m-1)}{2}}$ \\
\hline
\begin{tabular}{cc}Continuous Big $q$-Hermite: \\$H_n(x;a/q)$\end{tabular} &$(ae^{i\theta},ae^{-i\theta};q)_n $ &$(-a)^m{n\brack m}_qq^{\frac{m(m-1)}{2}} $\\
\hline
\begin{tabular}{cc}Continuous $q$-Laguerre: \\$P_n^{(\alpha)}(x/q)$\end{tabular} &   \begin{tabular}{cc}$\left(q^{\frac{1}{2}\alpha+\frac{1}{4}}e^{i\theta};q\right)_n$\\
$\times\left(q^{\frac{1}{2}\alpha+\frac{1}{4}}e^{-i\theta};q\right)_n$\end{tabular} &$(-1)^m{n\brack m}_qq^{\frac{m(m-1)}{2}}(q;q)_m(q^{m+\alpha+1};q)_{n-m}$\\
\hline
Little $q$-Laguerre/Wall: $P_n(x;a/q)$  &$x^n $&$(aq;q)_n(-1)^m{n\brack m}_qq^{\frac{m(m-1)}{2}} $\\
\hline
\multirow{2}*{$q$-Laguerre: $L_n^{(\alpha)}(x;q) $} & \multirow{2}*{$x^n $} & $q^{-n(\alpha+m)}q^{-\frac{n(n-1)}{2}}q^{\frac{m(m-1)}{2}}$\\
&  &$\times(-1)^m(q;q)_m (q^{\alpha+m+1};q)_{n-m}{n\brack m}_q$\\
\hline
Alternative $q$-Charlier: $K_n(x;a,q)$ & $x^n $& ${n\brack m}_q\frac{(-1)^mq^{\frac{m(m-1)}{2}}}{(-aq^m;q)_m(-aq^{2m+1};q)_{n-m}} $\\
\hline
$q$-Charlier: $C_n(q^{-x},a;q)$  &$(q^{-x};q)_n$ &$(-1)^{m+n}a^n{n\brack m}_qq^{\frac{m(m-1)}{2}-nm} $\\
\hline
Al-Salam-Carlitz I: $U_n^{(a)}(x;q) $ &$x^n(x^{-1};q)_n $&$a^{n-m}{n\brack m}_q $\\
\hline
Al-Salam-Carlitz II: $V_n^{(a)}(x;q)$ &$(x;q)_n $&$(a)^{n-m}q^{\frac{n(n-1)}{2}}q^{(m-n)(m-1)}(-1)^n{n\brack m}_q$\\
\hline
Continuous $q$-Hermite: $H_n(x/q)$  &$e^{-2in\theta} $&$e^{-im\theta}q^{n(m-1)}(-1)^{n+m}q^{\frac{n(n-1)}{2}}q^{\frac{m(m-1)}{2}}{n\brack m}_q $\\
\hline
Stieltjes-Wigert: $S_n(x;q)$ & $x^n $&$(q;q)_m(-1)^mq^{-\frac{n(n-1)}{2}}q^{\frac{m(m-1)}{2}}{n\brack m}_q $\\
\hline
Discrete $q$-Hermite I: $h_n(x;q) $ &$x^n(x^{-1};q)_n $&$(-1)^{m+n}{n\brack m}_q $\\
\hline
Discrete $q$-Hermite II: $\tilde{h}_n(x;q) $ &$(ix;q)_n $&$(-i)^mq^{n(m-1)}q^{\frac{n(n-1)}{2}}{n\brack m}_q $\\
\hline
\end{longtable}}}
\end{center}

\begin{center}
\footnotesize
{\renewcommand{\arraystretch}{1.9} 
{\setlength{\tabcolsep}{0.2cm} 
\begin{longtable}{|c c|}\caption{\textsc{connection coefficients in $q$-Askey scheme }} \label{connectionq-askey} \\
\hline
 $\textbf{Polynomials sets$\{P_n\}_{n\geq0}\rightarrow \{Q_m\}_{m\geq0}$}$ & $\textbf{Connection coefficients $C_m(n)$}$\\
 \hline
 \endfirsthead
 \begin{tabular}{cc}Askey-Wilson : \\ $P_n(x,a,b,c,d/q)\rightarrow P_m(x,a,\beta,\gamma,\delta/q)$ \end{tabular}
 &\begin{tabular}{cc} $\frac{a^{m-n}q^{m(m-n)}}{(q;q)_{n-m}}\frac{(ab,ac,ad,q;q)_n(abcdq^{n-1};q)_m}{(ab,ac,ad,q;q)_m(a \beta\gamma\delta q^{m-1};q)_m} $
  \\ $\times\hypergeomq{5}{4}{q^{m-n},a\beta q^m,a\gamma q^m,a\delta q^m,abcdq^{n+m-1}}{abq^m,acq^m,adq^m,a\beta\gamma\delta q^{2m}}{q}{q}$\end{tabular}\\
 \hline
 \begin{tabular}{ccc} $q$-Racah :  \\ $R_n(\mu(x);\alpha,\beta,\gamma,\delta/q)\rightarrow R_m(\mu(x);a,b,c,d/q)$ \\$\textrm{ provided that}\  \gamma\delta=cd$  \end{tabular}
 &\begin{tabular}{cc} $\frac{q^{m(m-n)}(q;q)_n}{(q;q)_m(q;q)_{n-m}}\frac{(\alpha\beta q^{n+1},aq,bd q,cq;q)_m}{(\alpha q,\beta\delta q,\gamma q,abq^{m+1};q)_m}$\\
$\times\hypergeomq{5}{4}{q^{m-n},\alpha\beta q^{m+n+1},aq^{m+1},bd q^{m+1},cq^{m+1}}{\alpha q^{m+1},\beta\delta q^{m+1},abq^{2m+2},\gamma q^{m+1}}{q}{q}$ \end{tabular}\\
  \hline
  \begin{tabular}{cc} Continous dual $q$-Hahn: \\ $P_n(x;a;b;c/q)\rightarrow P_m(x;a,\beta;\gamma/ q)$\end{tabular} &
   \begin{tabular}{cc} $\frac{a^{m-n}(q;q)_nq^{m(m-n)}(abq^m;acq^m;q)_{n-m}}{(q;q)_{n-m}(q;q)_m}$\\ $\times\hypergeomq{3}{2}{q^{m-n},a\beta q^{m},a\gamma q^m}{abq^m,acq^m}{q}{q} $\end{tabular}\\
  \hline
   \begin{tabular}{cc} Continous $q$-Hahn: \\ $P_n(x;a,b,c,d;q)\rightarrow P_m(x;a,\beta,\gamma,\delta;q)$ \end{tabular}
   & \begin{tabular}{cc} $\frac{q^{m(m-n)}(q;q)_n(abe^{2i\phi},ac,ad;q)_n(abcdq^{n-1};q)_ma^{m-n}e^{i\phi(m-n)}}
   {(q;q)_{n-m}(q;q)_m(abe^{2i\phi},ac,ad;q)_m(a\beta\gamma\delta q^{m-1};q)_m}$ \\ $\times\hypergeomq{5}{4}{q^{m-n},abcdq^{m+n-1},a\beta e^{2i\phi}q^m,a\gamma q^m,a\delta q^m}{abe^{2i\phi}q^m,acq^m,adq^m,a\beta\gamma\delta c q^{2m}}{q}{q}$\end{tabular}\\
  \hline
\begin{tabular}{cc}   Big $q$-Jacobi:  \\ $P_n(x;a,b,c;q)\rightarrow P_m(x;\alpha,\beta,\gamma;q)$\end{tabular}
 & \begin{tabular}{cc}   $\frac{q^{m(m-n)}(q;q)_n(abq^{n+1},\alpha q,\gamma q;q)_m}{(q;q)_m(q;q)_{n-m}(aq,cq,\alpha\beta q^{m+1};q)_m}$\\ $\times\hypergeomq{4}{3}{q^{m-n},abq^{m+n+1},\alpha q^{m+1},\gamma q^{m+1}}{aq^{m+1},cq^{m+1},\alpha\beta q^{2m+2}}{q}{q} $\end{tabular}\\
  \hline
  \begin{tabular}{cc} $q$-Hahn: \\$Q_n(q^{-x};\alpha,\beta,N/q)\rightarrow Q_m(q^{-x};\alpha_1,\beta_1,N_1/q)$ \end{tabular}
    &\begin{tabular}{cc} $\frac{q^{m(m-n)}(\alpha\beta q^{n+1},q^{-N_1},\alpha_1q;q)_m(q;q)_n}{(q;q)_m(q;q)_{n-m}(\alpha q,q^{-N},\alpha_1\beta_1 q^{m+1};q)_m}$\\
   $\times \hypergeomq{4}{3}{q^{m-n},\alpha\beta q^{n+m+1},q^{m-N_1},\alpha_1 q^{1+m}}{\alpha q^{m+1},q^{m-N},\alpha_1\beta_1 q^{2m+2}}{q}{q} $\end{tabular}\\
  \hline
 \begin{tabular}{cc} Dual $q$-Hahn: \\ $R_n(\mu(x);\gamma,\delta,N/q)\rightarrow R_m(\mu(x);\gamma_1,\delta_1,N_1/q)$\end{tabular} &\begin{tabular}{cc} $\frac{q^{m(m-n)}(q;q)_n(\gamma_1q,q^{-N_1};q)_m}{(q;q)_m(q;q)_{n-m}(\gamma q,q^{-N};q)_m}$\\ $\times\hypergeomq{3}{2}{q^{m-n},\gamma_1q^{m+1},q^{m-N_1}}{\gamma q^{m+1},q^{m-N}}{q}{q}$\end{tabular}\\
  \hline
 \begin{tabular}{cc} Al-Salam Chihara: \\$Q_n(x;a,b/q)\rightarrow Q_m(x;a,\beta/q)$\end{tabular} & $\frac{(q;q)_n\beta^{n-m}\left(\frac{b}{\beta};q\right)_{n-m}}{(q;q)_m(q;q)_{n-m}}\frac{(abq^m;q)_{n-m}}{(a\beta q^m;q)_{n-m}(a\beta;q)_{n-m}}$\\
  \hline
  \begin{tabular}{ccc}Continuous $q$-Jacobi \\$P_n^{(\alpha,\beta)}(x/q)\rightarrow P_m^{(\alpha,\beta_1)}(x/q) $ \\ with $\nu=\alpha+\beta+1$, $\lambda=\alpha+\beta_1+1$\end{tabular}
    & \begin{tabular}{cc}$\frac{q^{m(m-n)}(q^{\alpha+1};q)_n(q^{\nu+n},-q^{1/2(\lambda)},-q^{-1/2(\lambda+1)};q)_m}
    {(q;q)_{n-m}(q^{\lambda+m};q)_m(q^{\alpha+1},-q^{1/2(\nu)},-q^{-1/2(\nu+1)};q)_m} $\\
   $\times\hypergeomq{5}{4}{q^{m-n},q^{m+n+\nu},-q^{\frac{1}{2}(\lambda)+m},-q^{\frac{1}{2}(\lambda+1)+m},-q^{\alpha+m+1}}
   {q^{m+\alpha+1},-q^{1/2(\nu)+m},-q^{-1/2(\nu+1)+m},q^{2m+\lambda+1}}{q}{q}$\end{tabular}\\
 \hline
  \begin{tabular}{cc} Big $q$-Laguerre \\ $P_n(x;a,b;q)\rightarrow P_m(x;\alpha,\beta;q)$\end{tabular} &\begin{tabular}{cc} $\frac{q^{m(m-n)}(q;q)_n(\alpha q,\beta q;q)_m}{(q;q)_m(q;q)_{n-m}(aq,bq;q)_m}$\\ $\times\hypergeomq{3}{2}{q^{m-n},\alpha q^{m+1},\beta q^{m+1}}{aq^{m+1},b q^{m+1}}{q}{q} $\end{tabular}\\
  \hline
   \begin{tabular}{cc}Little $q$-Jacobi\\ $P_n(x;a,b/q)\rightarrow P_m(x;\alpha,\beta/q)$\end{tabular}
    &\begin{tabular}{cc} $\frac{q^{m(m-n)}(q;q)_n(abq^{n+1},\alpha;q)_m}{(q;q)_m(q;q)_{n-m}(aq,\alpha\beta q^{m+1};q)_m}$\\ $\times\hypergeomq{3}{2}{q^{m-n},abq^{m+n+1},\alpha q^{m+1}}{aq^{m+1},\alpha\beta q^{2m+2}}{q}{q} $\end{tabular}\\
  \hline
    \begin{tabular}{cc}$q$-Meixner \\ $M_n(q^{-x},b,c;q)\rightarrow M_m(q^{-x},\beta,\gamma;q)$\end{tabular}
     &\begin{tabular}{cc}$\left(\frac{\gamma}{c}\right)^m \frac{(q;q)_n(\beta q;q)_m}{(q;q)_m(q;q)_{n-m}(bq;q)_m}$\\ $\times\hypergeomq{2}{1}{q^{m-n},\beta q^{m+1}}{bq^{m+1}}{q}{\frac{\gamma}{c}q^{n-m+1}} $\end{tabular}\\
  \hline
   Quantum $q$-Krawtchouk & $\frac{(q;q)_n(q^{-N_1};q)_m}{(q;q)_m(q;q)_{n-m}(q^{-N};q)_m}\left(\frac{p}{p_1}\right)^m$\\
    $K_n(q^{-x};p,-N;q)\rightarrow K_m(q^{-x};p_1,-N_1;q)$ & $\times\hypergeomq{2}{1}{q^{n-m},q^{m-N_1}}{q^{m-N}}{q}{\frac{p}{p_1}q^{n-m+1}}$\\
  \hline
    $q$-Krawtchouk  & $q^{m(m-n)}\frac{(q;q)_n}{(q;q)_m(q;q)_{n-m}}\frac{(-pq^n,q^{-N};q)_m}{(-p_1q^m,q^{-N_1};q)_m}$\\
 $K_n(q^{-x};p,N;q)\rightarrow K_m(q^{-x};p_1,N_1:q)$  & $\times \hypergeomq{3}{2}{q^{m-n},-pq^{m+n},q^{m-N_1}}{q^{m-N},-p_1q^{2m+1}}{q}{q} $\\
  \hline
   \begin{tabular}{cc} Little $q$-Laguerre/Wall\\ $P_n(x;a/q)\rightarrow P_m(x;\alpha/q)$\end{tabular}
   & $(\alpha q)^{n-m}\frac{(q;q)_n}{(q;q)_m(q;q)_{n-m}}\frac{(\alpha q;q)_m}{(aq;q)_n}\left(\frac{a}{\alpha};q\right)_{n-m} $\\
  \hline
   \begin{tabular}{cc}$q$-Laguerre \\ $L_n^{(\alpha)}(x;q)\rightarrow L_m^{(\beta)}(x;q) $\end{tabular}
  & \begin{tabular}{cc} $\frac{q^{m(\alpha-\beta)}(q^{\alpha+m+1};q)_{n-m}}{(q;q)_{n-m}}$\\
  $\times \hypergeomq{2}{1}{q^{m-n},q^{\beta+m+1}}{q^{\alpha+m+1}}{q}{q^{1+n-m+\alpha-\beta}}$\end{tabular}\\
  \hline
      \begin{tabular}{cc}Alternative $q$-Charlier\\ $K_n(x;a;q)\rightarrow P_m(x;\alpha;q)$\end{tabular}
      &\begin{tabular}{cc} $\frac{(q;q)_nq^{m(m-n)}(-aq^n;q)_m}{(q;q)_m(q;q)_{n-m}(-\alpha q^{m};q)_n}$\\
     $\times \frac{(1+\alpha q^{2m})}{(1+\alpha q^{m+n})}\left(\frac{\alpha}{a}q^{m-n+1};q\right)_{n-m}\left(-aq^n\right)^{n-m} $\end{tabular}\\
  \hline
    \begin{tabular}{cc}$q$-Charlier \\ $C_n(q^{-x},a;q)\rightarrow C_m(q^{-x},\alpha;q)$\end{tabular} & $\frac{(q;q)_n}{(q;q)_m(q;q)_{n-m}}\left(\frac{\alpha}{a}q;q\right)_{n-m}\left(\frac{\alpha}{a}\right)^m$\\
      \hline
     \begin{tabular}{cc}Al-Salam Carlitz I \\ $U_n^{(a)}(x;q)\rightarrow U_m^{(\alpha)}(x;q)$\end{tabular}
     &\begin{tabular}{cc} $\frac{(q;q)_n}{(q;q)_m(q;q)_{n-m}}(-a)^{n-m}$\\ $\times q^{\frac{n(n-1)}{2}+\frac{m(m-1)}{2}}\left(\frac{\alpha}{a}q^{m-n};q\right)_{n-m}$\end{tabular}\\
  \hline
     \begin{tabular}{cc}Al-Salam Carlitz II \\$V_n^{(a)}(x;q)\rightarrow V_m^{(\alpha)}(x;q)$\end{tabular}
      &\begin{tabular}{cc} $\frac{(q;q)_n}{(q;q)_m(q;q)_{n-m}}(-a)^{n-m}q^{\frac{1}{2}[m(m-1)-n(n-1)]}$ \\$\times\left(\frac{\alpha}{a}q^{2m-1};q\right)_{n-m} $\end{tabular}\\
  \hline
\end{longtable}}}
\end{center} 

\begin{thebibliography}{10}

\bibitem{navima99}
I.~\textsc{Area, E. Godoy, A. Ronveaux and A. Zarzo}, \emph{Inversion problems
  in the $q$-{H}ahn tableau}, J. Symbolic Computation \textbf{28} (1999),
  767--776.

\bibitem{navima2001}
\bysame, \emph{Solving connection and linearization problems within the {A}skey
  {S}cheme and its $q$-analogue via inversion formulas}, J. Comput. Appl. Math.
  \textbf{133} (2001), 151--162.

\bibitem{navima2015}
I.~\textsc{Area, E. Godoy, J. Rodal, A. Ronveaux and A. Zarzo}, \emph{Bivariate
  {K}rawtchouk polynomials: {I}nversion and connection problems with the
  {NAVIMA} algorithm}, J. Comput. Appl. Math. \textbf{284} (2015), 50--57.

\bibitem{askeywilson79}
R.~\textsc{Askey and J. Wilson}, \emph{A set of orthogonal polynomials that
  generalize the {R}acah coefficients or 6-j symbols}, SIAM J. Math. Anal.
  \textbf{8} (1979), 1008--1016.

\bibitem{askeywilson85}
\bysame, \emph{Some basic hypergeometric polynomials that generalize {J}acobi
  polynomials}, Memoirs of the American Mathematical Society, vol. 319,
  Providence, Rhode Island,USA, 1985.

\bibitem{chaggara2005}
Y.~\textsc{Ben Cheikh and H. Chaggara}, \emph{Connection coefficients via
  lowering operators}, J. Comput. Appl. Math. \textbf{178} (2005), 45--61.

\bibitem{chaggara2006a}
\bysame, \emph{Connection coefficients between {B}oas-{B}uck polynomial sets},
  J. Math. Anal. Appl. \textbf{319} (2006), 665--689.

\bibitem{neila2015a}
N.~\textsc{Ben Romdhane}, \emph{A general theorem on inversion problems for
  polynomial sets}, Mediterr. J. Math. (2015), 1--11.

\bibitem{neila2015b}
N.~\textsc{Ben Romdhane and M. Gaied}, \emph{A generalization of the symmetric
  classical polynomials: {H}ermite and {G}egenbauer polynomials}, Integral
  Transforms Spec. Funct. (2016), To appear.

\bibitem{chaggara2007b}
H.~\textsc{Chaggara and W. Koepf}, \emph{Duplication coefficients via
  generating functions}, Complex Var. Elliptic Equ. \textbf{52} (2007),
  537--549.

\bibitem{mama2013}
M.~\textsc{Foupouagnigni, W. Koepf and D.D. Tcheutia}, \emph{Connection and
  linearization coefficients for the {A}skey {W}ilson polynomials}, J. Symbolic
  Comput. \textbf{53} (2013), 96--118.

\bibitem{mama2012}
M.~\textsc{Foupouagnigni, W. Koepf, D.D. Tcheutia and N.N. Sadjang},
  \emph{Representation of $q$-orthogonal polynomials}, J. Symbolic Comput.
  \textbf{47} (2012), 1347--1371.

\bibitem{gasper}
G.~\textsc{Gasper}, \emph{Projection formulas for orthogonal polynomials of a
  discrete variable}, J. Math. Anal. Appl. \textbf{45} (1974), 176--198.

\bibitem{ismail2005}
M.E.H. \textsc{Ismail}, \emph{Classical and quantum orthogonal polynomials in
  one variable}, Cambridge University Press, 2005.

\bibitem{koekoek2010}
R.~\textsc{Koekoek, P.A. Lesky and R.F. Swarrtow}, \emph{Hypergeometric
  orthogonal polynomials and their $q$-analogues}, Springer-Verlag, Berlin,
  2010.

\bibitem{lamiri2013}
I.~\textsc{Lamiri and A. Ouni}, \emph{$d$-{O}rthogonality of some basic
  hypergeometric polynomials}, Georgian Math. J. \textbf{20} (2013), 729--751.

\bibitem{lewanowicz2013}
S.~\textsc{Lewanowicz}, \emph{Construction of recurrences for the coefficients
  of expansions in $q$-classical orthogonal polynomials}, J. Comput. Appl.
  Math. \textbf{153} (2003), 295--309.

\bibitem{maroni1989}
P.~\textsc{Maroni}, \emph{L'orthogonalit\'{e} et les r\'{e}currences de
  polyn$\hat{o}$mes d'ordre sup\'{e}rieur \`{a} deux}, Ann. Fac. Sci. Toulouse
  \textbf{10} (1989), 105--139.

\bibitem{Rain}
E.D. \textsc{Rainville}, \emph{Special functions}, The Macmillan Company, New
  York, 1960.

\bibitem{rota1973}
G.C. \textsc{Rota, D. Kahaner and A. Odlyzko}, \emph{On the foundations of
  combinatoric theory {VIII.} {F}inite operator calculus}, J. Math. Anal. Appl.
  \textbf{42} (1973), 684--760.

\bibitem{dehesa2001}
J.~\textsc{S\'{a}nchez-Ruiz and J.S. Dehesa}, \emph{Some connection and
  linearization problems for polynomials in and beyond the {A}skey {S}cheme},
  J. Comput. Appl. Math. \textbf{133} (2001), 579--591.

\bibitem{mama2015}
D.D. \textsc{Tcheutia, M. Foupouagnigni, W. Koepf and N.N. Sadjang},
  \emph{Coefficients of multiplication formulas for classical orthogonal
  polynomials}, Ramanujan J. (2015), 1--35.

\bibitem{verma2}
A.~\textsc{Verma}, \emph{Certain expansions of the basic hypergeometric
  functions}, Math. Comp. \textbf{20} (1966), 151--157.

\end{thebibliography}

\end{document}